\documentclass{amsart}
\usepackage{amsmath,latexsym}
\usepackage[psamsfonts]{amssymb}
\usepackage{times}
\usepackage[mathcal]{euscript}
\numberwithin{equation}{section}

\renewcommand{\epsilon}{\varepsilon}

\newcommand{\be}{\begin{equation}}
\newcommand{\ee}{\end{equation}}
\newcommand{\no}{\nonumber}

\newcommand{\N}{\mathbb{N}}

\newcommand{\Q}{\mathbb{Q}}
\newcommand{\R}{\mathbb{R}}

\newcommand{\T}{\mathbb{T}}

\newcommand{\Z}{\mathbb{Z}}

\newcommand{\cE}{{\mathcal E}}
\newcommand{\cF}{{\mathcal F}}

\newcommand{\supp}{{\ensuremath{\mathrm{supp}}}}

\renewcommand{\det}{\mathop{\mathrm{det}}}

{\bf}{\it}
\newtheorem{theorem}{Theorem}[section]
\newtheorem{lemma}[theorem]{Lemma}
\newtheorem{corollary}[theorem]{Corollary}

\newtheorem{remark}[theorem]{Remark}

\begin{document}

\title[Existence of isolated band in a system of three
particles in an optical lattice ]{Existence of an isolated band in a
system of three particles in an optical lattice }

\author{Gianfausto Dell'Antonio, Saidakhmat N. Lakaev,
  Ahmad M. Khalkhuzhaev}



\begin{abstract}
We prove the existence of two-and three-particle bound states of the
Schr\"{o}dinger operators $h_\mu(k),k\in \T^d$ and $H_\mu(K),K\in
\T^d$ associated to Hamiltonians $\mathrm{h}_{\mu}$ and
$\mathrm{H}_{\mu}$ of a system of two and three identical bosons on
the lattice $\Z^d, d=1,2$ interacting via pairwise zero-range
attractive $\mu<0$ or repulsive $\mu>0$ potentials. As a consequence
we show the existence of an isolated band in the two and three
bosonic systems in an optical lattice.
\end{abstract}

\maketitle


\section{Introduction}
Throughout physics, stable composite objects are usually formed by
way of attractive forces, which allow the constituents to lower
their energy by binding together. Repulsive forces separate
particles in free space. However, in structured environment such as
a periodic potential and in the absence of dissipation, stable
composite objects can exist even for repulsive interactions that
arises from the lattice band structure \cite{Nature}.

The Bose-Hubburd model which is used to describe the repulsive pairs
is the theoretical basis for applications. The work \cite{Nature}
exemplifies the important correspondence between the Bose-Hubburd
model \cite{ Bloch,Jaksch_Zoller} and atoms in optical lattices, and
helps pave the way for many more interesting developments and
applications.

The main goal of the paper is to prove the existence of
three-bosonic bound states of the Schr\"{o}dinger operator
$H_\mu(K),K\in \T^d$ for the cases of attractive $\mu<0$ and
repulsive $\mu>0$ interactions.

Cold atoms loaded in an optical lattice provide a realization of a
quantum lattice gas. The periodicity of the potential gives rise to
a band structure for the dynamics of the atoms.

The dynamics of the ultracold atoms loaded in the lower or upper
band is well described by the Bose-Hubburd hamiltonian
\cite{Nature}; we give in section 2 \ the corresponding
Schr\"{o}dinger operator.

In the continuum case due to rotational invariance the Hamiltonian
separates in a free hamiltonian for the center of mass  and in a
hamiltonian $H_{\mathrm{rel}} $ for the relative motion. Bound
states are eigenstates of $H_{\mathrm{rel}}$.

The kinematics of the quantum particles on the lattice is rather
exotic. The discrete laplacian \it is not rotationally invariant \rm
and therefore one cannot separate the motion of the center of mass.

One can rather resort to a Floquet-Bloch decomposition. The
three-particle Hilbert space $ \mathcal{H} \equiv \ell ^2 [({
\Z}^d)^3] $ is represented as direct  integral associated to the
representation of the discrete group $ {\Z} ^d $ by shift operators
\begin{equation*}
\ell ^2[({\Z}^d)^3] = \int_{K\in {\T}^d} \oplus \ell
^2[({\Z}^d)^2]\eta(dK),
\end{equation*}
where $\eta(dp)=\frac {d^dp}{(2\pi)^d}$ is the (normalized) Haar
measure on the torus $\T^d$. Hence the total three-body hamiltonian
appears to be decomposable
\begin{equation*}\label{decompose}
\mathrm{H}=\int\limits_{\T^d}\oplus H(K)\eta(dK).
\end{equation*}

The fiber hamiltonian $H(K) $ depends perimetrically on the {\it
quasi momentum} $ K \in \T^d \equiv \R^d / (2 \pi {\Z}^d) .$ It is
the sum of a free part and an interaction term, both bounded and the
dependence on $K$ of the free part is continuous.

Bound states $\psi_{E,K}$ are solution of the Schr\"{o}dinger
equation
\begin{equation*}
H(K) \psi_{E,K} = E \psi_{E,K},  \,\, \psi_{E,K} \in \ell
^2[(\Z^d)^2].
\end{equation*}
It is known that in dimension $d=3$ for the case we are considering
the hamiltonian $H(K),K\in \T^3$ for $K=0\in \T^3$ has infinitely
many bound states (the Efimov effect)\cite{ALKh12,LSN93}. Since, the
operator $H(K),K\in \T^3$ continuously depends on $K\in \T^3$ one
can conclude that there exists a neighborhood $
\mathbb{G}_0\subset\T^3$ of $0\in \T^3$ and for all $K\in
\mathbb{G}_0$ the operator $H(K),K\in \mathbb{G}_0$ has bound states
\cite{ALzM04,LSN93}.

In the paper we study the hamiltonian $\mathrm{H}_{\mu}$ of a system
of three bosons on the lattice $ {\Z}^d , \, d=1,2 $ interacting
through attractive $\mu<0$ or repulsive $\mu>0$ zero-range potential
$\mu V$.

Remark that the $d=1,2$ situation is of interest, since in the
experiment it corresponds to a low depth of the lattice along one
direction, whilst the lattice in the perpendicular directions
remains very deep \cite{Nature}.

Our main new results is to prove the existence of three-bosonic
bound states of the Scar\"o\-din\-ger operator $H_\mu(K)$, $K\in
\T^d$ associated to $\mathrm{H}_{\mu}$, with energy lying below the
bottom resp. above the top of the essential spectrum for the case of
attractive $\mu<0$ resp. repulsive $\mu>0$ interaction.

As a consequence we show the existence of a band spectrum of the
hamiltonian $\mathrm{H}_{\mu}$ of a system of three bosons.

We can conclude that these results for a three bosonic system
theoretically predicts the existence stable attractively and
repulsively bound objects of three  atoms. Hopefully, this also can
be experimentally confirmed as it done for atoms pair with repulsive
interaction in \cite{Nature}.

To our knowledge analogous results have not been published yet even
for a system of three particles interacting via attractive
potentials on Euclidean space $\R^d$.

We remark that the same formalism could be used to prove the
existence of at least one bound state with energy lying below resp.
above the essential spectrum for the case of attractive resp.
repulsive short range potentials.

This paper is organized as follows.

Section 1 is introduction. In section 2 we give explicitly the
hamiltonian of the two-body and three-body case in the
Schr\"{o}dinger representation. It corresponds to the Hubburd
hamiltonian in the number of particles representation. In section 3
we introduce the Floquet-Bloch decomposition (von Neumann
decomposition) and choose relative coordinates to describe
explicitly the discrete Schr\"{o}dinger operator $H_\mu(K) , \; K
\in \T^d $. In section 4 we state our main results. In section 5 we
introduce \it channel operators \rm and describe the essential
spectrum of $ H_\mu(K),\;K\in \T^d $ by means discrete spectrum of $
h_\mu(k), \; k \in \T^d $.We prove the existence of bound states in
section 6.

\section{Hamiltonians of three identical
bosons on lattices in the coordinate and momentum representations}
Let ${\Z}^{d},d=1,2$ be the $d$-dimensional lattice. Let
$\ell^2[({\Z}^{d})^{m}],d=1,2$ be Hilbert space of square-summate
functions $\,\,\hat{\varphi}$ defined on the Cartesian power
$({\Z}^{d})^{m},d=1,2$ and let $\ell^{2,s}[(\Z^{d})^{m}]\subset
\ell^{2}[(\Z^{d})^{m}]$ be the subspace of functions symmetric with
respect to the permutation of coordinates of the particles.

Let $\Delta$ be the lattice Laplacian, i.e., the operator which
describes the transport of a particle from one site to another site:
$$ (\Delta\hat{\psi})(x)=-\sum_{\mid s\mid =1} [
\hat{\psi}(x)-\hat{\psi}(x+s)],\quad \hat{\psi}\in\ell^2({\Z}^d).$$

 The free hamiltonian $\hat{\mathrm{h}}_0$ of a system of two
identical quantum mechanical particles with mass $m=1$ on the
$d$-dimensional lattice $\Z^d,d=1,2$ in the coordinate
representation is associated to the self-adjoint operator
$\hat{\mathrm{h}}_0$ in the Hilbert space
$\,\,\ell^{2,s}[(\Z^{d})^{2}]$:
\begin{equation}\no
\hat{\mathrm{h}}_0=\Delta\otimes I + I\otimes\Delta.
\end{equation}

 The total hamiltonian $\hat {\mathrm{h}}_\mu$ of a system of two quantum-mechanical
 identical particles  with the two-particle pairwise zero-range attractive interaction
$\mu\hat v$ is a bounded perturbation of the free hamiltonian
$\hat{\mathrm{h}}_0$ on the Hilbert space $\ell^{2,s}[( {\Z}^d)^2]$:
\begin{equation*}\label{two-part}
\hat{\mathrm{h}}_\mu =\hat{\mathrm{h}}_0+\mu\hat v.
\end{equation*}
Here $\mu\in\R$ is coupling constant and
\begin{equation*}
(\hat v\hat \psi)(x_1,x_2) = \delta _{x_1 x_2}
{\hat\psi}(x_1,x_2),\quad {\hat\psi} \in \ell^{2,s}[({\Z}^d)^2],
\end{equation*}
where $\delta _{x_1 x_2}$ is  the Frolicker delta.

Analogously, the free hamiltonian $\hat{\mathrm{H}}_0$ of a system
of three identical bosons with mass $m=1$ on the $d$-dimensional
lattice $\Z^d$ is defined on $\ell^{2,s}[({\Z}^d)^3]$:
\begin{equation*}\label{free0}
\widehat H_0=\Delta\otimes I\otimes I + I \otimes \Delta \otimes I +
I\otimes I\otimes \Delta.
\end{equation*}
The total  hamiltonian $ \hat{\mathrm{H}}_\mu $  of a system of
three quantum-mechanical identical particles with  pairwise
zero-range interaction $\hat v=\hat v_{\alpha}=\hat
v_{\beta\gamma},\alpha,\beta,\gamma=1,2,3$ is a bounded perturbation
of the free hamiltonian $\hat {\mathrm{H}}_0$:
\begin{equation*}\label{total}
 \hat {\mathrm{H}}_\mu=\hat {\mathrm{H}}_0+\mu\hat {\mathbb{V}},
\end{equation*}
where $\hat{\mathbb{V}}=\sum\limits_{i=1}^{3} \hat
{V}_{\alpha},\,{V}_{\alpha}=\hat V, \alpha=1,2,3 $ is the
multiplication operator on $\ell^{2,s}[({\Z}^d)^3]$ defined by
\begin{align*}
&(\hat V_{\alpha}\hat\psi)(x_1,x_2,x_3)=  \delta _{x_\beta
x_\gamma}\hat\psi(x_1,x_2,x_3),\\
&\alpha\prec\beta\prec\gamma\prec\alpha,\,\alpha,\beta,\gamma=1,2,3,\,\,
\hat\psi\in \ell^{2,s}[({\Z}^d)^3].\\
\end{align*}
\subsection{The momentum  representation}
 Let ${\T}^d=(-\pi,\pi]^d$ be the $d-$ dimensional torus and
$L^{2,s}[(\mathbb{T}^{d})^{m}]\subset L^2[(\mathbb{T}^{d})^{m}]$ be
the subspace of symmetric functions defined on the  Cartesian power
${({\T}^d)^m},\,m\in \N$.

Let  \begin{equation*}
 \hat \Delta={\cF}\Delta{\cF}^{*}
 \end{equation*} be Fourier transform of the Laplacian $\Delta$, where
\begin{equation*} \label{eq-0.9}
\cF: \ \ell^2(\Z^d) \to L^2(\T^d), \quad [\cF(f)](p) \ := \ \sum_{x
\in \Z^d} e^{-\mathrm{i} (p, x)} \:  f(x)
\end{equation*}
is the standard Fourier transform with the inverse
\begin{equation*} \label{eq-0.10}
\cF^*: \ L^2(\T^d) \to \ell^2(\Z^d), \quad [\cF^*(\psi)](x) \ := \
\int_{ \T^d} e^{\mathrm{i}(p, x)} \: \psi(p) \, \eta(dp),
\end{equation*}
and $\eta(dp) =\frac {d^dp}{(2\pi)^d}$ is the (normalized) Haar
measure on the torus.  It easily checked that $\hat \Delta$ is the
multiplication operator by the function $ \varepsilon(\cdot)$, i.e.
\begin{equation*}
(\hat \Delta f)(k)= {\varepsilon}(p)f(p),\quad f \in L^2(\T^d),
\end{equation*}
where
\begin{equation*}
\varepsilon(p)=2\sum_{i=1}^{d}(1-\cos p^{(i)}),\quad p=(p^{(1)},...,
p^{(d)})\in \T^d.
\end{equation*}
The two-particle total hamiltonian  $\mathrm{h}_\mu$ in the momentum
representation is given on $L^{2,s}[({\T}^d)^2]$ as follows
\begin{equation*}
\mathrm{h}_\mu = \mathrm{h}_0+ \mu v.
 \end{equation*}
Here the free hamiltonian $\mathrm{h}_0$ is of the form
\begin{equation*}
\mathrm{h}_0 =\hat \Delta \otimes I+I \otimes \hat \Delta,
\end{equation*}
where $I$ is identity operator on $L^2(\T^d)$ and $\otimes$ denotes
the tensor product. It is easy to see that the operator
$\mathrm{h}_0$ is the multiplication operator by the function $
\varepsilon (k_1)+\varepsilon (k_2):$
\begin{equation*}
(\mathrm{h}_0f)(k_1 ,k_2)=[\varepsilon(k_1)+\varepsilon (k_2 )]f(k_1
,k_2),\,\,f\in L^{2,s}[({\T}^d)^2],
\end{equation*}
 The  integral operator $v$ is the convolution type
\begin{align*}
&( vf)(k_1,k_2 )= {\int\limits_{({\T}^d)^2} }
 \delta (k_1
+k_2 -k_1'-k_2')f(k_1',k_2')\eta(dk_1')\eta(dk_2')\\
&=\int\limits_{\T^d} f(k_1',k_1 +k_2-k_1')\eta(dk_1'),\, f\in
L^{2,s}[({\T}^d)^2],
\end{align*}
where $ \delta (\cdot)$ is the $d-$ dimensional Dirac delta
function.

The three-particle hamiltonian in the momentum representation is
given by the bounded self-adjoint operator on the Hilbert space
 $L^{2,s}[({\T}^d)^3]$ as follows
\begin{equation*}
\mathrm{H}_\mu= \mathrm{H}_0+\mu({V}_{1}+V_{2}+V_{3}),
\end{equation*}
where  $\mathrm{H}_0$ is of the form
\begin{equation*}
\mathrm{H}_0=\hat \Delta\otimes I\otimes I+I \otimes
 \hat \Delta \otimes I+I\otimes I\otimes  \hat \Delta,
\end{equation*}
i.e., the free hamiltonian $\mathrm{H}_0$ is the multiplication
operator by the function $\sum\limits_{\alpha=1}^{3} \varepsilon
(k_\alpha)$:
\begin{equation*}
(\mathrm{H}_0f)(k_1,k_2,k_3)  =
[\sum\limits_{\alpha=1}^{3}\varepsilon (k_\alpha)]f(k_1,k_2,k_3),
\end{equation*}
and

\begin{align*}
&( V_{\alpha}f)(k_{\alpha},k_{\beta},k_{\gamma})\\
&={\int\limits_{({\T}^d)^3} } \delta (k_{\alpha} -k_{\alpha}')\,
\delta (k_{\beta} +k_{\gamma}
-k_{\alpha}'-k_{\beta}')f(k'_1,k'_2,k'_3)\eta(dk_{\alpha}')\eta(dk_{\beta}')\eta(dk_{\gamma}')\\
&={\int\limits_{{\T}^d}}f(k_{\alpha},k_{\beta}',k_{\beta}
+k_{\gamma}-k_{\beta}')\eta(dk_{\beta}'), \quad f\in
L^{2,s}[({\T}^d)^3].
\end{align*}

\section{Decomposition of the hamiltonians into von Neumann direct integrals.
Quasi-momentum and coordinate systems}

Denote by $k=k_1+k_2\in \T^d$ resp. $K=k_1+k_2+k_3\in \T^d$ the {\it
two-} resp. {\it three-particle quasi-momentum} and define the set
$\mathbb{Q}_{k}$ resp. $\mathbb{Q}_{K}$ as follows
\begin{align*}
&\mathbb{Q}_{k}=\{(k_1,k-k_1){\in }({\T}^d)^2: k_1 \in\T^d,
k-k_1\in\T^d\}
\end{align*}
resp.
\begin{align*}
&\mathbb{Q}_{K}=\{(k_1,k_2,K-k_1-k_2){\in }({\T}^d)^3: k_1,k_2
\in\T^d, K-k_1-k_2\in\T^d\}.
\end{align*}
We introduce the mapping
\begin{equation*}
\pi_{1}:(\T^d)^2\to \T^d,\quad \pi_1(k_1, k_2)=k_1
\end{equation*}
resp.
\begin{equation*}
\pi_2:(\T^d)^3\to (\T^d)^2,\quad \pi_2(k_1, k_2, k_3)=(k_1, k_2).
\end{equation*}

Denote by  $\pi_k$, $k\in \T^d$ resp. $\pi_{K}$ , $K\in \T^d$ the
restriction of $\pi_1$ resp. $\pi_2$ onto $\Q_{k}\subset (\T^d)^2,$
resp. $\Q_{K}\subset (\T^d)^3$, that is,
\begin{equation*}\label{project}\pi_{k}= \pi_1\vert_{\Q_{k}}\quad
\text{and}\quad \pi_{K}=\pi_{2}\vert_{\Q_{K}}.
\end{equation*}
At this point it is useful to remark that $ \Q_{k},\,\, k \in {\T}^d
$ resp.
 $ \Q_{K},\,\, K \in
{\T}^d $ are $d-$ resp. $2d-$ dimensional manifold isomorphic to
${\T}^d$ resp.
 ${({\T}^d)^2}$.
\begin{lemma}\label{unitary}
The map $\pi_{k}$, $k\in \T^d$ resp. $\pi_{K}$, $K\in \T^d$ is
bijective from $\mathbb{Q}_{k}\subset (\T^d)^2$ resp.
$\mathbb{Q}_{K}\subset (\T^3)^3$ onto  $\T^d$ resp. $(\T^d)^2$ with
the inverse map given by
\begin{equation*}
(\pi_{k})^{-1}(k_1)=(k_1,k-k_1)
\end{equation*}
resp.
\begin{equation*}
(\pi_{K})^{-1}(k_1,k_2)= (k_1, k_2, K-k_1-k_2).
\end{equation*}
\end{lemma} Let $L^{2,e}({\T}^d)\subset L^2 ( {\T}^d)$ be the
subspace of even functions. Decomposing the Hilbert space $
L^{2,s}[(\T^d)^2]$ resp. $ L^{2,s}[({\T}^d)^3]$  into the direct
integral
\begin{equation*}\label{tensor}
L^{2,s}[({\T}^d)^2] = \int_{k\in {\T}^d} \oplus
L^{2,e}({\T}^d)\eta(dk)
\end{equation*}
resp.
\begin{equation*}
L^{2,s}[({\T}^d)^3] = \int_{K\in {\T}^d} \oplus
L^{2,s}[({\T}^d)^2]\eta(dK)
\end{equation*}
yields the  decomposition of the hamiltonian $\mathrm{h}_\mu$ resp.
$\mathrm{H}_\mu$ into the direct integral
\begin{equation}\label{fiber2}
\mathrm{h}_\mu= \int\limits_{k \in \T^d}\oplus\tilde
h_\mu(k)\eta(dk)
\end{equation}
resp.
\begin{equation}\label{fiber3}
\mathrm{H}_\mu=\int\limits_ {K \in {\T}^d}\oplus \tilde
H_\mu(K)\eta(dK).
 \end{equation}
\subsection{The discrete Schr\"{o}dinger operators} The fiber operator $\tilde h_\mu(k),$
$k \in {\T}^d$ from the direct integral decomposition \eqref{fiber2}
acting in $L^2[\mathbb{Q}_{k}]$ according to Lemma \ref{unitary} is
unitarily equivalent to the operator $h_\mu(k),$ $k \in {\T}^d$
given by
\begin{equation}\label{two} h_\mu(k) =h_0(k)+\mu v.
\end{equation}
The operator $h_0(k)$ is multiplication operator by the function
$\cE_k(p)$:
\begin{equation*}(h_0(k)f)(p)=\cE_k(p)f(p),\quad f \in L^{2,e}(\T^d),
\end{equation*}
where
\begin{equation*}\label{E-alpha}
\cE_{k}(p)= \varepsilon (k-p) +\varepsilon
 (p)
\end{equation*}
and
\begin{equation*}
(vf)(p)= \int\limits_{{\T}^d}f(q)d \eta (q), \quad f \in
L^{2,e}({\T}^d).
\end{equation*}

The fiber operator $\tilde H_\mu(K),$ \,$K \in {\T}^d$ from the
direct integral decomposition \eqref{fiber3} acting in
$L^2[\mathbb{Q}_{K}]$ according to Lemma \ref{unitary} is unitarily
equivalent to the operator $H_\mu (K),$ $K \in {\T}^d$ given by
\begin{equation*}\label{three-particle}
H_\mu(K)=H_0(K)+\mu (V_1+V_2+ V_3).
\end{equation*}

The operator $H_0(K)$ acts in the Hilbert space
$L^{2,s}[({\T}^d)^2]$ and has form
\begin{equation*}\label{TotalK}
(H_0(K)f)(p,q)=E(K;p,q)f(p,q),\quad f\in  L^{2,s}[({\T}^d )^2],
\end{equation*}
where
\begin{equation*} E(K;p,q)= \varepsilon
(K-p-q)+\varepsilon(q) + \varepsilon(p).
\end{equation*}
The  operator $\mathbb{V}=V_1+V_2+V_3$ acting on $
L^{2,s}[({\T}^d)^2]$ in coordinates $(p,q)\in (\T^d)^2$ can be
written in the form
\begin{align*}\label{potential}
(\mathbb{V}f)(p,q)&=\int\limits_{\T^d}f(p,t)
\eta(dt)+\int\limits_{\T^d}f(t,q)\eta(dt)+\int\limits_{\T^d}f(t,K-p-q)\eta(dt).
\end{align*}
\section{Statement of the main results}

According to the Weyl theorem \cite{RSIV} the essential  spectrum
   $\sigma_{\mathrm{ess}}(h_\mu(k))$ of the
operator $h_\mu(k),\,k \in \T^d$ coincides with the spectrum $
{\sigma}( h_0 (k) ) $ of $h_0(k).$ More specifically,
$$
\sigma_{\mathrm{ess}}(h_\mu(k))= [\cE_{\min}(k) ,\,\cE_{\max}(k)],
$$
where \begin{align*} &\cE_{\min}(k)\equiv\min_{p\in
\T^d}\cE_k(p)=2\sum_{i=1}^{d}[1-\cos(\frac{k^{(i)}}{2})]\\
&\cE_{\max}(k)\equiv\max_{p\in\T^d}
\cE_{k}(p)=2\sum_{i=1}^{d}[1+\cos(\frac{k^{(i)}}{2})].
\end{align*}

Since for each $K\in\T^d$ the function $E(K;p,q)$ is continuous on
$(\T^d)^2,\,d=1,2$ there exist
\begin{align*}
{E}_{\min }(K)\equiv\min_{p,q\in \T^d}E(K;p,q),\quad {E}_{\max
}(K)\equiv\max_{p,q \in \T^d}E(K;p,q)
\end{align*}
and $\sigma(H_0(K))= [E_{\min }(K) ,\,E_{\min }(K)]$.

\begin{remark}
We remark that the essential spectrum
$[\cE_{\min}(k),\cE_{\max}(k)]$ strongly depends on the
quasi-momentum $k\in\T^d;$ when $k=\vec{\pi}=(\pi,...,\pi)\in \T^d$
the essential spectrum of $h_{\mu}(k)$ degenerated to the set
consisting of a unique point $\{\cE_{\min}(\vec{\pi})=
\cE_{\max}(\vec{\pi}) =2d\}$ and hence the essential spectrum of
$h_{\mu}(k)$ is not absolutely continuous for all $k\in \T^d.$ The
similar arguments are true for the essential spectrum of $H_0(K)$.
\end{remark}

The following theorem asserts the existence of eigenvalues of the
operator $h_\mu(k)$ and can be proven in the same way as Theorem
\ref{existencetwo} in \cite{LKhL12,LU12}.
\begin{theorem}\label{existencetwo}
 For any $\mu<0$ resp. $\mu>0$  and $k\in \T^d,\,d=1,2$\, the
operator $h_\mu(k)$ has a unique eigenvalue $e_{\mu}(k)$, which is
even in $k\in \T^d$ and satisfies the
relations:
$$e_{\mu}(k)<\cE_{\min}(k),\,k\in\T^d \,\, \mbox{and} \,\,
e_{\mu}(0)<e_{\mu}(k), k\in\T^d\setminus\{0\}
 \,\, \mbox{for}\,\, \mu<0$$
 resp.
$$e_{\mu}(k)>\cE_{\max}(k),\,k\in\T^d\,\, \mbox{and}\,\,
e_{\mu}(0)>e_{\mu}(k), k\in \T^d\setminus\{0\} \,\, \mbox{for}\,\,
\mu>0.$$ The eigenvalue $e_{\mu}(k)$ is holomorphic function in
$k\in\T^d$. For any $k\in\T^d$ the associated eigenfunction
$f_{\mu,e_{\mu}(k)}(p)$ is holomorphic in $p\in{\T}^d$ and has the
form
\begin{equation*}\label{eigen}
f_{\mu,e_{\mu}(k)}(\cdot)=\frac{\mu
c(k)}{\cE_{k}(\cdot)-e_{\mu}(k)},\mu<0, \,\,\mbox{resp.}\,\,
f_{\mu,e_{\mu}(k)}(\cdot)=\frac{\mu
c(k)}{e_{\mu}(k)-\cE_{k}(\cdot)},\mu>0
\end{equation*}
where $c(k)\neq 0$ is  a normalizing constant. Moreover, the vector
valued mapping
\begin{equation*} f_{\mu}:\mathbb{\T}^d \rightarrow
L^2[\mathbb{\T}^d,\eta(dk);L^{2}({\T}^d)],\,k\rightarrow
f_{\mu,e_{\mu}},\mu<0
\end{equation*}
resp.
\begin{equation*} f_{\mu}:\mathbb{\T}^d \rightarrow
L^2[\mathbb{\T}^d,\eta(dk);L^{2}({\T}^d)],\,k\rightarrow
f_{\mu,e_{\mu}},\mu>0
\end{equation*}
is holomorphic on $\mathbb{\T}^d$.
\end{theorem}

In the next theorem the essential spectrum of the three-particle
operator $H_\mu(K),\,K \in \T^d$ is described by the spectra of the
non perturbed operator $H_0(K)$ and the discrete spectrum of the
two-particle operator $h_\mu(k),\,k \in \T^d.$

\begin{theorem}\label{ess} Let $d=1,2.$ For
any $\mu\ne0$ the essential spectrum $\sigma_{\mathrm{ess}}(H_\mu
(K))$ of $H_\mu(K)$ satisfies the following equality
$$
\sigma_{\mathrm{ess}}(H_\mu (K))=\cup _{k\in
{\T}^d}\{e_\mu(K-k)+\varepsilon (k)\} \cup [E_{\min }(K) ,\,E_{\min
}(K)],
$$
where $e_\mu(k)$ is the unique eigenvalue of the operator
$h_\mu(k),k\in \T^d$.
\end{theorem}

The main result of the paper is given in following theorem, which
states the existence of  bound states of the three-particle operator
$H_\mu(K),\,K \in \T^d$.
\begin{theorem}\label{existencethree} Let $d=1,2$. \item[(i)]
For all $\mu<0$ and $K \in \T^d$  the operator $H_\mu(K)$ has an
eigenvalue  $E_\mu(K)$ lying below the bottom
$\tau^{b}_{\mathrm{ess}}(H_\mu (K))$ of the essential spectrum. The
eigenvalue $E_\mu(K)$ is a holomorphic function in $K\in \T^d$.
Moreover, the associated eigenfunction (bound state)
$f_{\mu,E_\mu(K)}(\cdot,\cdot)\in L^{2,s}[({\T}^d)^2]$ is
holomorphic in $(p,q)\in({\T}^d)^2$ and the vector valued mapping
\begin{equation*}\label{map}
f_{\mu}:\mathbb{\T}^d \rightarrow
L^2[\mathbb{\T}^d,\eta(dK);L^{2,s}[({\T}^d)^2]],\,k\rightarrow
f_{\mu,E_\mu(K)}
\end{equation*} is also holomorphic in $K\in\mathbb{\T}^d$.
\item[(ii)]
For all $\mu>0$ and $K \in \T^d$  the operator $H_\mu(K)$ has an
eigenvalue  $E_\mu(K)$ lying above the top
$\tau^{t}_{\mathrm{ess}}(H_\mu (K))$ of the essential spectrum. The
eigenvalue $E_\mu(K)$ is a holomorphic function in $K\in \T^d$ and
the associated eigenfunction (bound state)
$f_{\mu,E_\mu(K)}(\cdot,\cdot)\in L^{2,s}[({\T}^d)^2]$ is
holomorphic in $(p,q)\in({\T}^d)^2$.  Moreover, the vector valued
mapping
\begin{equation*}\label{map}
f_{\mu}:\mathbb{\T}^d \rightarrow
L^2[\mathbb{\T}^d,\eta(dK);L^{2,s}[({\T}^d)^2]],\,k\rightarrow
f_{\mu,E_\mu(K)}
\end{equation*} is also holomorphic in $K\in\mathbb{\T}^d$.
\item[(iii)] For all $\mu<0$ resp. $\mu>0$ and $K \in \T^d$  the operator $H_\mu(K)$ has
no eigenvalue  lying above the top resp. below the bottom of the
essential spectrum.
\end{theorem}
Theorem \ref{existencethree} yields a corollary, which asserts the
existence of a band spectrum for a system of two resp. three
interacting bosons on the lattice $\Z^d, d=1,2$.
\begin{corollary}\label{existencebound} For any $|\mu|>0$ the
two-particle resp.\ three-particle hamiltonian $\mathrm{h_\mu}$
resp. $\mathrm{H_\mu}$ has a band spectrum
\begin{equation*}[\min_{k} e_\mu(k),\max_{k} e_\mu(k)]\,\, \mbox{resp.}\,\,
[\min_{K} E_\mu(K),\max_{K} E_\mu(K)].
\end{equation*}
\end{corollary}
\begin{remark}\label{Rem_exist}
We remark that for large $|\mu|>0$ the two-particle resp.\
three-particle hamiltonian $\mathrm{h_\mu}$ resp. $\mathrm{H_\mu}$
has an isolated  band
\begin{equation*}[\min_{k} e_\mu(k),\max_{k} e_\mu(k)]\,\, \mbox{resp.}\,\,
[\min_{K} E_\mu(K),\max_{K} E_\mu(K)].
\end{equation*}
\end{remark}

Let
\begin{equation*}\label{esstwo}
\sigma_{\mathrm{esstwo}}(H_\mu (K))=\cup _{k\in
{\T}^d}\{e_\mu(K-k)+\varepsilon (k)\}
\end{equation*}
be the two-particle part and
\begin{equation*}
\tau^{b}_{\mathrm{ess}}(H_{\mu}(K))=\inf
\sigma_{\mathrm{ess}}(H_\mu(K))\,\,\mbox{resp.}\,\,
\tau^{t}_{\mathrm{ess}}(H_{\mu}(K))=\sup \sigma_{\mathrm{ess}}(H_\mu
(K))
\end{equation*} be the bottom resp. top of the essential spectrum of the Schr\"{o}dinger operator $H_\mu(K)$.
\begin{remark}\label{Rem_exist} For any $\mu<0$ Theorems \ref{existencetwo} and \ref{ess}
yield that the relations
\begin{equation*}\label{d=1ord=2}
\sigma_{\mathrm{esstwo}}(H_\mu (K))\neq\emptyset
\end{equation*}
and hence
\begin{equation*}\label{d=1or d= 2}
\tau^{b}_{\mathrm{ess}}(H_{\mu}(K))<E_{\min}(K)
\end{equation*}
hold, which allows the existence of bound states of three
attractively interacting bosons on lattice $\Z^d$
\cite{ALzM04,LSN93}.
\end{remark}

\begin{remark}\label{Rem_exist_efim} We remark that for the three-particle Schr\"{o}dinger
operator $H_{\mu}(K),K\in\T^3$ associated to a system of three
bosons in the lattice $\Z^3$, there exists $\mu_0<0$ such that
\begin{equation*}\label{Efimov}
\tau^{b}_{\mathrm{ess}}(H_{\mu_0}(0))=E_{\min}(0)
\end{equation*}
and for all $\mu<\mu_0<0$
\begin{equation}\label{infiniteness}
\tau^{b}_{\mathrm{ess}}(H_{\mu}(0))< E_{\min}(0).
\end{equation}
At the same time for any  $0\neq K\in\T^3$ the following relations
\begin{equation*}\label{d=1ord=2}
\sigma_{\mathrm{esstwo}}(H_{\mu_0}(K))\neq\emptyset
\end{equation*}
and
\begin{equation}\label{finiteness}
\tau^{b}_{\mathrm{ess}}(H_{\mu}(K))< E_{\min}(K).
\end{equation}
hold. Thus, only the operator $H_{\mu_0}(0)$ has infinitely many
number of eigenvalues below the bottom of the essential spectrum
(the Efimov effect) \cite{ALzM04,LSN93} and this result yields the
existence of bound states of $H_{\mu_0}(K),K \in \mathbb{G}_0$,
where $\mathbb{G}_0 \subset \T^d$ is a neighborhood of the point
$0\in \T^d$. Moreover for any nonzero $K \in \T^3$ the operator
$H_{\mu_0}(K)$ has only finitely many number of bound states.
\end{remark}
\begin{remark}
Note that analogous remarks on the existence of bound states of
$H_\mu(K)$ for the case of repulsive $\mu>0$ are valid.
\end{remark}
\begin{remark}
The results for the attractive interactions are characteristic to
the Schr\"{o}dinger operators associated to a system of three
particles moving on the lattice $\Z^d$ and the Euclidean space
$\R^d$ in dimension $d=1,2$.
\end{remark}
\section{The essential spectrum of the operator $ H_\mu(K)$}
Let $\mu\ne 0$. Since the particles are identical there is only one
channel operator
 $ H_{\mu,ch} (K),$ \ $K \in \T^d,\,d=1,2$ defined in the
Hilbert space $L^{2}[({\T}^d )^2]$ as
\begin{equation*}
 H_{\mu,ch}(K)=H_0(K)+\mu V.
\end{equation*}
The operators $H_0(K)$ and $V=V_\alpha$ act as follows
\begin{equation*}\label{TotalK}
(H_0(K)f)(p,q)=E (K;p,q)f(p,q),\quad f\in L^{2}[({\T}^d )^2],
\end{equation*}
where
\begin{equation*}\label{Eps}
E(K;p,q)= \varepsilon (K-p-q)+\varepsilon(q) +\varepsilon(p)
\end{equation*}
and
 \begin{equation*}\label{Poten}(V
f)(p,q)= \int\limits_{\T^d}f(p,t)\eta(dt),\quad f\in
L^{2}[({\T}^d)^2].
\end{equation*}

The decomposition of the space $L^{2}[(\T^d)^2]$ into the direct
integral

 $$L^{2}[({\T}^d )^2]= \int\limits_{k\in \T^d}
\oplus L^{2}( \T^d) \eta(dp)
$$
yields for the operator $H_{\mu,ch}(K)$ the decomposition
 $$H_{\mu,ch}(K)=\int\limits_{k\in \T^d}
 \oplus h_{\mu}(K,p) \eta(dp).$$
The fiber operator $h_{\mu}(K,p)$ acts in the Hilbert space
$L^{2}(\T^d)$ and has the form
\begin{equation}\label{representation}
 h_{\mu}(K,p) ={h}_{\mu}
(K-p)+\varepsilon(p) I,
\end{equation} where $I_{L^{2}(\T^d)}$ is
the identity operator and the operator ${h}_{\mu}(p)$ is unitarily
equivalent to ${h}_{\mu}(p),$ defined by \eqref{two}. The
representation \eqref{representation} of the operator $h_{\mu}(K,p)$
and Theorem \ref{existencetwo} yield the following description for
the spectrum of $h_{\mu}(K,p)$
\begin{align}\label{stucture}
 &\sigma (h_{\mu}(K,p))
 =[e_\mu(K-p)+\varepsilon(p)]\cup [E_{\min }(K) ,\,E_{\min }(K)].
\end{align}
\begin{lemma}\label{inequality}  For any $K\in \T^d$ the bottom $\tau^{b}_{\mathrm{ess}}(H_{\mu}(K))$
resp. top $\tau^{t}_{\mathrm{ess}}(H_{\mu}(K))$ of the essential
spectrum  satisfies the inequality
$$\tau^{b}_{\mathrm{ess}}(H_{\mu}(K))<E_{\min} (K)$$ resp.
$$\tau^{t}_{\mathrm{ess}}(H_{\mu}(K))>E_{\max}(K).$$
\end{lemma}
\begin{proof}
For any $\mu\ne0$ and $K\in \T^d$ we define $Z_\mu (K,p)$ on
$\T^d,d=1,2$ by
\begin{equation}\label{defZ}
 Z_\mu(K,p)=e_\mu (K-p)+\varepsilon (p).
\end{equation} Theorem \ref{existencetwo} yields the inequality
\begin{align*}
&Z_\mu(K,K-p_{\min}(K))= e_{\mu}(K-p_{\min}(K))+
\varepsilon(p_{\min}(K))\\
&<\cE_{\min}(K-p_{\min}(K))+\varepsilon(p_{\min}(K))=E_{\min}(K),
\end{align*}
where $(p_{\min}(K),p_{\min}(K))\in (\T^d)^2$ is a minimum point of
the function $E(K;p,q)$. The definition of
$\tau^{b}_{\mathrm{ess}}(H_{\mu}(K))$ gives
\begin{align*}
&\tau^{b}_{\mathrm{ess}}(H_{\mu}(K))=\inf_{p\in\T^d} Z_\mu(K,p) \leq
e_\mu(K-p_{\min}(K))+ \varepsilon(p_{\min}(K))<E_{\min}(K),
\end{align*} which proves Lemma \ref{inequality}.
For the case $\mu>0$ the proof of Lemma \ref{inequality} is
analogously.
\end{proof}
\section{Proof of the main results}
Set
\begin{align*}
&E_{\min}(K,k)=\min_{q}E(K,k\,;q)=\min_{q}\cE_{K-k}(q)+\varepsilon(k),\\
&E_{\max}(K,k)=\max_{q}E(K,k\,;q)=\max_{q}\cE_{K-k}(q)+\varepsilon(k).
\end{align*}

For any $\mu\in \R$ and $K,k\in \T^d,d=1,2$ the determinant
$\Delta_\mu (K,k\,;z)$ associated to the operator $h_{\mu}(K,k)$ can
be defined as real-analytic function in $\mathrm{C}\setminus
[E_{\min }(K,k),\,E_{\max }(K,k)]$ by
\begin{equation*}\label{determinant}
\Delta_\mu (K,k\,; z ) = 1+\mu \int\limits_{\T^d}\frac{\eta(dq)}{E
(K; k\,,q)-z}.
\end{equation*}
\begin{lemma}\label{nollar2}
For any $\mu\in \R$ and  $K,k\in\T^d$ the number  $z \in
{\mathrm{C}} {\setminus} [E_{\min }(K,k),\,E_{\max}(K,k)]$ is an
eigenvalue of the operator $h_{\mu}(K,k) $ if and only if
$$
 \Delta_\mu (K, k\,; z) = 0.
$$
 \end{lemma}
The proof of Lemma \ref{nollar2}  is simple and can be found in
\cite{LKhL12}.
\begin{remark}
We note that for each $\mu<0$ resp. $\mu>0$ and $K,k \in \T^d$ there
exist $z_1=z_1(K,k)<E_{\min }(K,k)$ resp.
$z_2=z_2(K,k)>E_{\max}(K,k)$ such that for all $z\leq z_1$ resp.
$z\geq z_2$ the function $\Delta_\mu (K,k\,;z)$ is non-negative and
the square root function $\Delta^{\frac{1}{2}}_\mu (K,k\,;z)$ is
well defined.
\end{remark}
We define for each $\mu \in\R$ and $z\in \R\setminus
\big[\tau^b_{\mathrm{ess}}(H_{\mu}(K)),\tau^t_{\mathrm{ess}}(H_{\mu}(K))\big]$
the self-adjoint compact operator $\mathrm{L}_\mu( K,z),\,K\in \T^d$
as
\begin{equation}\label{compact_operator}
[\mathrm{L}_\mu(K,z)\psi](p)=2\mu
 \int\limits_{\T^d} \frac{\Delta^{-\frac{1}{2}}_\mu(K,p,z)
\Delta^{-\frac{1}{2}}_\mu(K,q, z)}{E(K;p,q)-z}\psi(q)
\eta(dq),\\
\psi\in L^2(\T^d).
\end{equation}
Notice that for $\mu<0$ the operator $\mathrm{L}_\mu(K,z),\, z <
\tau^b_{\mathrm{ess}}(H_{\mu}(K))$ has been introduced in
\cite{LSN93} to investigate Efimov's effect for the three-particle
lattice Schr\"{o}dinger operator $H_\mu(K)$ associated to a system
of three bosons on the lattice $\Z^3$.

\begin{lemma}\label{eigenvalue}
Let  $\mu\ne0$ and
$z\in\R\setminus[\tau^b_{\mathrm{ess}}(H_{\mu}(K)),\tau^t_{\mathrm{ess}}(H_{\mu}(K)]$.

\item[(i)] If $f\in L^{2,s}[({\T}^d)^2]$ solves $H_\mu(K)f = zf$, then
$$\psi(p)=\Delta^{\frac12}_\mu(K,p\:;z)\int\limits_{\T^d}f(p,t)\eta(dq)\in L^2({\T}^d)$$ solves
$\mathrm{L_\mu}(K,z)\psi=-\psi$.

\item[(ii)] If $\psi \in L^2({\T}^d)$ solves $\mathrm{L_\mu}(K,z)\psi=-\psi$, then
\begin{equation*}
f(p,q)=\dfrac{-\mu
[\varphi(p)+\varphi(q)+\varphi(K-p-q)]}{E(K;p,q)-z}\in
L^{2,s}[({\T}^d)^2]
\end{equation*}
solves the equation $H_\mu(K)f = zf$, where
$\varphi(p)=\Delta^{-\frac12}_\mu(K,p\:;z)\psi(p)$.
\end{lemma}

\begin{proof}
\item[(i)]
 We prove Lemma \ref{eigenvalue} for the case $\mu<0$ and $z<\tau^{b}_{\mathrm{ess}}(H_{\mu}(K))$. The case $\mu>0$ and
$z>\tau^t_{\mathrm{ess}}(H_{\mu}(K))$ can be proven analogously.
\item[(i)]Let $\mu<0$. Assume that for some $K\in\T^d$ and $z<\tau^{b}_{\mathrm{ess}}(H_{\mu}(K))$
the Schr\"{o}dinger equation $$(H_{\mu}(K)f)(p,q)=z f(p,q),$$ i.e.,
the equation
\begin{align}\label{t1}
&[E(K;p,q)-z]f(p,q)\\
&=-\mu[\int\limits_{\T^d}f(p,t)\eta(dt)+\int\limits_{\T^d}f(t,q)\eta(dt)+\int\limits_{\T^d}f(K-p-q,t)\eta(dt)]\nonumber
\end{align}
has solution $f\in L^{2,s}[({\T}^d)^2]$. Denoting by
$\varphi(p)=\int\limits_{\T^d}f(p,t)\eta(dt) \in L^2({\T}^d)$ we
rewrite equation \eqref{t1} as follows
\begin{equation}\label{solution}
f(p,q)=-\mu\frac{\varphi(p)+\varphi(q)+\varphi(K-p-q)}{E(K;p,q)-z},
\end{equation}
which gives for $\varphi \in L^2({\T}^d)$ the equation
\begin{equation}\label{BSequation}
\varphi(p)=-\mu\int\limits_{\T^d}
\dfrac{\varphi(p)+\varphi(t)+\varphi(K-p-t)}{E(K;p,t)-z}\eta(dt).
\end{equation}
Since the function $E(K;p,t)$ is invariant under the changing
$K-p-t\rightarrow t$ of variables we have
\begin{align}
\varphi(p)\left[1+\mu\int\limits_{\T^d}
\dfrac{\eta(dq)}{E(K;p,q)-z}\right]=-2
\mu\int\limits_{\T^d}\dfrac{\varphi(q)}{E(K;p,q)-z}\eta(dq).\notag
\end{align}
Denoting by $\psi(p)=\Delta^{\frac12}_\mu(K,p\:;z)\varphi(p)$ and
taking into account that
$\Delta_\mu(K,p\:;z)\neq0,\,z<\tau^{b}_{\mathrm{ess}}(H_{\mu}(K))$
we get the equation
\begin{equation}\label{B-S}
\mathrm{L_\mu}(K,z)\psi=-\psi
\end{equation}\label{function}

\item[(ii)] Assume that $\psi$ is a solution of equation
\eqref{B-S}.Then the function
\begin{equation}\label{function}
\varphi(p)=\Delta^{-\frac12}_\mu (K,p\,;z)\psi(p)
\end{equation}
is a solution of equation \eqref{BSequation} and hence the function
defined by \eqref{solution} is a solution of the equation $H_\mu(K)f
= zf$, i.e., is an eigenfunction of the operator $H_{\mu}(K)$
associated to the eigenvalue
$z<\tau^{b}_{\mathrm{ess}}(H_{\mu}(K)).$
\end{proof}
{\bf Proof of Theorem \ref{ess}.} Theorem \ref{ess} can be proven,
applying equality \eqref{stucture},
 by the same way as Theorem 3.2
in \cite{ALKh12,ALzM04}. $\square$

{\bf Proof of Theorem \ref{existencethree}.}
$(i)$

 Let $\mu<0$ and $\mathrm{L_\mu}(K,z),K\in \T^d, d=1,2$ be the
operator defined in \eqref{compact_operator}. Then for any non-zero
$f\in L^2(\T^d)$ and $z<\tau^{b}_{\mathrm{ess}}(H_{\mu}(K))$ the
following relations
\begin{align}\label{norm}
&(\mathrm{L_\mu}(K,z)f,f)=-2\mu\int\limits_{\T^d}
\int\limits_{\T^d}\frac{f(p)\overline{f(q)}\eta(dp)\eta(dq)}
{\Delta^{\frac{1}{2}}_\mu(K,p,z)
\Delta^{\frac{1}{2}}_\mu(K,q,z)(E(K;p,q)-z)}\\ \nonumber
&>\frac{-2\mu}{E_{\max}-z} \Big |\int\limits_{\T^d}
\frac{f(p)\eta(dp)} {\Delta^{\frac{1}{2}}_\mu(K,p,z) }\Big |^2
\geq0\nonumber
\end{align}
hold. Let
\begin{equation*}F_z(f)=\int\limits_{\T^d} \frac{ f(p)\eta(dp) }
{\Delta^{\frac{1}{2}}_\mu(K,p,z)},\,z<\tau^{b}_{\mathrm{ess}}(H_{\mu}(K))
\end{equation*}
be bounded linear functional defined on $L^2(\T^d)$. According the
Riesz theorem
$$||F_z||=[\int\limits_{\T^d} \frac{\eta(dp)}
{\Delta_\mu(K,p,z)}]^{\frac{1}{2}}.$$ Let $\mathfrak{M}_+\subset
L^2(\T^d)$ subset of all non-negative functions. Then
\begin{equation}
||F_z||=\sup_{f\in L^2(\T^d) }\frac{|F_z|}{|f|}\geq \sup_{f\in
\mathfrak{M}_+}\frac{|F_z|}{|f|}\geq\frac{|F_z{\psi}_z|}{|{\psi}_z|}=||{\psi}_z||,
\end{equation}
where ${\psi}_z=[\Delta^{\frac{1}{2}}_\mu(K,p,z)]^{-1}$. Hence we
have that $||F_z||=||{\psi}_z||$.

Since for any $p\in \T^d$ the function $\Delta_\mu(K,p,z)$ is
monotone decreasing in\\
$z\in(-\infty,\tau^{b}_{\mathrm{ess}}(H_{\mu}(K)))$ there exists
a.e. the point-wise limit
$$\lim_{z \to
\tau^{b}_{\mathrm{ess}}(H_{\mu}(K))}\frac{1}{\Delta_\mu(K,p\,;z)}
=\frac{1}{\Delta_\mu(K,p\,;\tau^{b}_{\mathrm{ess}}(H_{\mu}(K)))}.$$
The Fatou's theorem yields the inequality
$$\int\limits_{\T^d} \frac{\eta(dp)
}{\Delta_\mu(K,p\,;\tau^{b}_{\mathrm{ess}}(H_{\mu}(K)))}\leq
\liminf_{z\to\tau^{b}_{\mathrm{ess}} (H_{\mu}(K))}\int\limits_{\T^d}
\frac{\eta(dp) }{\Delta_\mu(K,p\,;z)}.$$ 

Let $p_\mu(K)\in \T^d,\,K \in \T^d$ be a minimum point of the
function $Z_\mu(K,p),\,K \in \T^d$ defined in \eqref{defZ}. Then
$Z_\mu(K,p)$ has the following asymptotics
\begin{equation}\label{Z}
Z_\mu(K,p)=\tau^{b}_{\mathrm{ess}}(H_{\mu}(K))+
(B(K)(p-p_\mu(K)),p-p_\mu(K)) +o(|p-p_\mu(K)|^2),
\end{equation}
as $|p-p_\mu(K)| \to 0$, where $B(K)$ is non-negative matrix. For
any $K,p \in \T^d$ there exists a $\gamma=\gamma(K,p)>0$
neighborhood $W_\gamma(Z_\mu(K,p))$ of the point $Z_\mu(K,p)\in
\mathrm{C}$ such that for all $z \in W_\gamma(Z_\mu(K,p))$ the
following equality holds
\begin{equation}\label{expansion}
\Delta_\mu (K,p,z)=\sum_{n=1}^{\infty}C_n(\mu,
K,p)[z-Z_\mu(K,p)]^n,\\
\end{equation}
where
\begin{align*}
&C_1(\mu,K,p)=-\mu \int\limits_{\T^d}
\dfrac{\eta(dq)}{[E(K;p,q)-Z_\mu(K,p)]^2}>
0.\\
\end{align*}

According to \eqref{expansion} for any $K\in \T^d$ there is
$U_{\delta(K)}(p_{\mu}(K))$ such that for all $p\in
U_{\delta(K)}(p_{\mu}(K))$ the equality
\begin{align}\label{nondeger}
&\Delta_\mu(K,p,\tau^{b}_{\mathrm{ess}}(H_{\mu}(K)))\\ \nonumber
&=(Z_\mu(K,p)-\tau^{b}_{\mathrm{ess}}(H_{\mu}(K)))\hat
\Delta_\mu(K,p,\tau^{b}_{\mathrm{ess}}(H_{\mu}(K)))
 \end{align}
holds. Putting \eqref{Z} into \eqref{nondeger} yields the estimate
\begin{equation*}\label{otsenka2}
\Delta_\mu (K,p,\tau^{b}_{\mathrm{ess}}(H_{\mu}(K)))\leq
M(K)|p-p_\mu(K)|^2.
\end{equation*}

Hence, we have
$$\int\limits_{\T^d} \frac{\eta(dp)
}{\Delta_\mu(K,p\,;\tau^{b}_{\mathrm{ess}}(H_{\mu}(K)))}=+\infty.$$
Consequently, for any  $P>0$ there exists $z_0 <
\tau^{b}_{\mathrm{ess}}(H_{\mu}(K)) $ such that the inequality
\begin{equation}\label{norma}
||F_{z_0}||=\supp_{||\psi||=1}(F_{z_0}\psi,\psi)=[\int\limits_{\T^d}
\frac{\eta(dp)} {\Delta_\mu(K,p,z)}]^{\frac{1}{2}}>P
\end{equation} holds. Since
for all $z\leq \tau^{b}_{\mathrm{ess}}(H_{\mu}(K))$ the positive
function $(E_{\max}-z)^{-1}$ is bounded above, the inequalities
\eqref{norm} and \eqref{norma} yield the existence $\psi \in
L^2(\T^d),\,||\psi||_{L^2(\T^d)}=1$ such that the relation
$(\mathrm{L}_{\mu}(K,z_0)\psi,\psi)>1$ holds. At the same time
$(\mathrm{L}_{\mu}(K,z)\psi,\psi)$ is continuous in $z\in
(-\infty,\tau^{b}_{\mathrm{ess}}(H_{\mu}(K))$ and
$$(\mathrm{L}_{\mu}(K,z)\psi,\psi)\rightarrow 0\,\, \mbox{as}\,\, z\rightarrow
-\infty.$$ Thus, there exists real number
$E_\mu(K)$,\,$-\infty<E_\mu(K)<z_0<\tau^{b}_{\mathrm{ess}}(H_{\mu}(K))$
that satisfies the equality
$$(\mathrm{L}_{\mu}(K,E_\mu(K))\psi,\psi)=1$$ and
hence the Hilbert-Schmidt theorem implies that the equation
\begin{equation*}\label{ holomorphic_solution}
\mathrm{L}_{\mu}(K,E_\mu(K))\ \psi=\psi
\end{equation*}
has a solution $\psi \in L^2(\T^d),\,||\psi||=1$.

Lemma \ref{eigenvalue} yields that
$E_\mu(K)<\tau^{b}_{\mathrm{ess}}(H_{\mu}(K))$ is an eigenvalue of
the operator $H_{\mu}(K)$ and the associated eigenfunction
$f_{E_\mu(K)}(K;p,q)$ takes the form
\begin{equation}\label{eigenfunction}
f_{E_\mu(K)}(K;p,q)=\dfrac{-\mu
c(K)[\varphi(p)+\varphi(q)+\varphi(K-p-q)]}{E(K;p,q)-E_\mu(K)}\in
L^{2,s}[({\T}^d)^2]
\end{equation}
with $c(K)= ||f_{E_\mu(K)}(K;p,q)||^{-1},\,K \in \T^d$ being the
normalizing constant.

Since for any $K\in\T^d$ the functions $\Delta_\mu(K,p\,;E_\mu(K))$
and $E(K;p,q)-E_\mu(K)>0$ are holomorphic in $p,q \in \T^d$ the
solution $\psi$ of equation \eqref{B-S} and the function $\varphi$
given in \eqref{function} are holomorphic in $p\in \T^d$. Hence, the
eigenfunction \eqref{eigenfunction} of the operator $H_{\mu}(K)$
associated to the eigenvalue
$E_\mu(K)<\tau^{b}_{\mathrm{ess}}(H_{\mu}(K))$ is also holomorphic
in $p,q\in \T^d$.

For any $z<\tau^{b}_{\mathrm{ess}}(H_{\mu}(K))$ the kernel function
$$\mathrm{L}_{\mu}(K,z;p,q)=-2\mu \frac{\Delta^{-\frac{1}{2}}_\mu(K,p,z)
\Delta^{-\frac{1}{2}}_\mu(K,q, z)}{E(K;p,q)-z}$$ of the compact
self-adjoint operator $\mathrm{L}_{\mu}(K,z)$ is holomorphic in $p,q
\in \T^d$. The Fredholm determinant
$D_\mu(K,z)=\det[I-\mathrm{L}_{\mu}(K,z)]$ associated to the kernel
function is real-analytic  function in $z\in
(-\infty,\tau^{b}_{\mathrm{ess}}(H_{\mu}(K)))$. Lemma
\ref{eigenvalue} and the Fredholm theorem yield that each eigenvalue
of the operator $H_{\mu}(K)$ is a zero of the determinant
$D_\mu(K,z)$ and vice versa. Consequently, the compactness of the
torus $\T^d$ and the implicit function theorem give that the
eigenvalue $E_\mu(K)$ of $H_{\mu}(K)$ is a holomorphic function in
$K \in \T^d,\,d=1,2$.

Since for any $p,q\in\T^d$ the functions
$\Delta_\mu(K,p\:;E_\mu(K))$ and $E(K;p,q)-E_\mu(K)$ are holomorphic
in $K\in \T^d$ the solution $\psi$ of \eqref{B-S} and the function
$\varphi$ defined by \eqref{function} are holomorphic in $K\in
\T^d$. Hence, the eigenfunction \eqref{eigenfunction} of the
operator $H_{\mu}(K)$ associated to the eigenvalue
$E_\mu(K)<\tau^{b}_{\mathrm{ess}}(H_{\mu}(K))$ is also holomorphic
in $K\in \T^d$. Consequently, the vector valued mapping
\begin{equation*}\label{map}
f_{\mu}:\mathbb{T}^d \rightarrow
L^2[\mathbb{T}^d,\eta(dK);L^{2,s}[({\T}^d)^2]],\,K\rightarrow
f_{\mu,K}(\cdot,\cdot)
\end{equation*} is holomorphic in $\mathbb{T}^d$.

$(ii)$ This part of Theorem \ref{existencethree} can be proven
similarly using the corresponding Lemmas.

$(iii)$ An application  the minimax principle to the operator
$H_\mu(K)$ completes the proof of Theorem \ref{existencethree}.

{\bf Acknowledgement.} The authors thank I.A.Ikromov  for useful
discussions and remarks.The first author would like to thank
Fulbright program for supporting the research project during
2013-2014 academic year. The work was supported by the Grant
F4-FA-F079 of Fundamental Science Foundation of Uzbekistan.

\end{document}